\documentclass[12pt]{amsart}

\usepackage{amssymb,mathrsfs} 
\usepackage{graphicx}

\usepackage{color} %

\newtheorem{theorem}{Theorem}

\newtheorem{proposition}[theorem]{Proposition}
\newtheorem{observation}[theorem]{Observation}
\newtheorem{corollary}[theorem]{Corollary}

\newtheorem{problem}{Problem}

\newtheorem{them}{Theorem}
\newtheorem{lema}[them]{Lemma}

\theoremstyle{definition}

\newtheorem{remark}[theorem]{Remark}

\begin{document}

\title[Unchanging Roman domination]{Roman domination in graphs: the class  $\mathcal{R}_\mathbf{UV R}$}

\author{Vladimir Samodivkin}
\address{Department of Mathematics, UACEG, Sofia, Bulgaria}
\email{vl.samodivkin@gmail.com}
\today
\keywords{Roman domination number; Roman bondage; domination; tree}

\begin{abstract}
For a graph $G = (V, E)$, a Roman dominating function $f : V \rightarrow \{0, 1, 2\}$ 
has the property that every vertex $v \in V $with $f (v) = 0$ has a neighbor $u$ with 
$f (u) = 2$. The weight of a Roman dominating function $f$ is the sum $f (V) = \cup_{v\in V} f (v)$, and
the minimum weight of a Roman dominating function on $G$ is the Roman domination
number $\gamma_R(G)$  of $G$. 
The Roman bondage number $b_R(G)$ of $G$ is the minimum cardinality of all sets 
$F \subseteq E$ for which $\gamma_R(G - F) > \gamma_R(G)$.
A graph $G$ is in the class $\mathcal{R}_{UVR}$ if 
the Roman domination number remains unchanged when a vertex is deleted.
In this paper we obtain tight upper bounds for $\gamma_R(G)$ 
and $b_R(G)$ provided a graph $G$ is in $\mathcal{R}_{UVR}$.  
We present necessary and sufficient conditions for 
a tree to be in the class $\mathcal{R}_{UV R}$. We  give 
a constructive characterization of $\mathcal{R}_{UVR}$-trees 
using labellings. 
\end{abstract}

\maketitle

{\bf  MSC 2010}: 05C69


\large

\section{Introduction} 

All graphs considered in this paper are finite, undirected, loopless, and without multiple edges. We
refer the reader to the book \cite{w} for graph theory notation and terminology not described in this paper.
We denote the vertex set and the edge set of a graph $G$ by $V(G)$ and $ E(G),$  respectively. 
 We write $K_n$ for the {\em complete graph} of order $n$ and $C_n$ for a {\em cycle} of length $n$. 
Let $P_m$	 denote the {\em path} with $m$ vertices. 
In a graph $G$, for a subset $S \subseteq V(G)$ the {\em subgraph induced} by $S$ is the graph $\left\langle S\right\rangle$ 
with vertex set $S$ and edge set $\{xy \in E(G)\colon x,y \in S\}$. 
The {\em complement} $\overline{G}$ of $G$ is the simple graph
whose vertex set is $V(G)$ and whose edges are the pairs of
nonadjacent vertices of $G$. 
The {\em join}  of simple graphs $G$ and $H$, written $G \vee H$, is the 
graph obtained from the disjoint union of $G$ and $H$ by adding the edges $\{xy \mid x\in V(G), \ y \in V(H)\}$. 
	For any vertex $x$ of a graph $G$,  $N_G(x)$ denotes the set of all  neighbors of $x$ in $G$,  
	$N_G[x] = N_G(x) \cup \{x\}$ and the {\it degree} of $x$ is $deg_G(x) = |N_G(x)|$. 
	The {\em minimum} and {\em maximum} degree of a graph $G$ are denoted by $\delta(G)$ and $\Delta(G)$, respectively.
   For a graph $G$, let $x \in X \subseteq V(G)$. 
   A vertex $y \in V(G)$ is a $X$-{\it private neighbor}  of $x$ if $N_G[y] \cap X = \{x\}$. 
		The  $X$-{\it private neighborhood}  of $x$, denoted $pn_G[x,X]$, is  the set of all $X$-private neighbors of $x$.

A set $D \subseteq V(G)$ is a {\em dominating set} if for each $x \in
V(G)$ either $v \in D$ or $x$ is adjacent to some $y \in D$. 
 The {\em domination number} $\gamma(G)$ is the minimum cardinality of a
dominating set of $G$, and a dominating set $D$ of minimum
cardinality is called a $\gamma$-$set$ of $G$.
If $D$ is a $\gamma$-set of $G$, then $pn[v, S] \not = \emptyset$ for each $v \in D$. 
An efficient dominating set  in a graph $G$ is a set $S \subseteq V(G)$
 such that $\{N[s] \mid  s \in S\}$  is a partition of $V (G)$.
 All efficient dominating sets in graph $G$ have the same cardinality that is equal to $\gamma(G)$ ( \cite{hhs1}). 
 The concept of domination in graphs has many applications to several fields. 
 Domination naturally arises in facility location problems,  
 in problems involving finding sets of representatives, in monitoring communication
or electrical networks, and in land surveying.  Many variants of the basic concepts
of domination have appeared in the literature. We refer to \cite{hhs1,hhs2}  for a survey of the area.

A variation of domination called Roman domination was introduced independently by Arquilla and Fredricksen \cite{af},
ReVelle  \cite{re1,re2} and Stewart \cite{s}, which was motivated with the following legend.
 In the $4$th century A.D., Emperor Constantine the Great issued a decree to ensure the protection of
the Roman empire. Constantine ordered that each city in the empire either has a legion
stationed within it for defense or lies near a city with two standing legions. This way,
if a defenseless city were attacked, a nearby city could dispatch reinforcements without
leaving itself defenseless. The natural problem is to determine how few legions suffice
to protect the empire. The concept of Roman domination can be formulated in terms of
graphs.  More formally, following Cockayne et al. \cite{cd04}, a {\em Roman dominating  function} ({\em RDF}) on a
graph $G$ is a vertex labeling $f : V(G) \rightarrow \{0, 1, 2\}$
 such that every vertex with label $0$ has a neighbor with label $2$. 
 For an RDF $f$, let $ V_i^f = \{v \in V (G) : f(v) = i\}$ for i = 0, 1, 2. 
 Since this partition determines $f$, we can equivalently write 
 $f=(V_0^f; V_1^f; V_2^f)$.
 The weight $f(V(G))$ of an RDF $f$ on $G$ is the value $\Sigma_{v\in V(G)} f (v)$, 
  which equals $|V_1^f| + 2|V_2^f|$. 
The {\em Roman domination number} of a graph $G$, denoted by $\gamma_R(G)$, 
is the minimum weight of a Roman dominating function on $G$. 
 Thus, $\gamma_R(G)$ is the minimum number of legions needed to protect cities whose
adjacency graph is $G$.
A function $f=(V_0^f; V_1^f; V_2^f)$  is called a $\gamma_R$-{\em function} on $G$, 
if it is a Roman dominating function and $f(V (G)) = \gamma_R(G)$.  
If $f$ is an RDF on a graph $G$ and $H$ is a subgraph of $G$, 
then we denote the restriction of $f$ on $H$ by $f|_H$.

Cockayne et al.\cite{cd04} showed that
\begin{equation}\label{eqq}\gamma(G)\le \gamma_{R}(G)\le 2\gamma(G).\end{equation}
A graph $G$ is called to be {\it Roman} if $\gamma_{\rm R}(G)=2\gamma(G)$.
All Roman paths and cycles are $P_{3k}, C_{3k}, P_{3k+2}$, and $C_{3k+2}$ (\cite{cd04}). 
Liedloff et al. \cite{lklpo} and Liu and Chang \cite{lc} 
investigated algorithmic aspect of Roman domination.
 Applications of Roman domination were also shown in \cite{ck10}.
Also see ReVelle and Rosing \cite{rer} for an integer programming
formulation of the problem.

The concept of  Roman bondage in graphs was introduced by Jafari Rad and Volkmann in \cite{rv0}. 
Let $G$ be a graph with maximum degree at least two. The Roman bondage number
$b_R(G)$ of $G$ is the minimum cardinality of all sets $E_1 \subseteq E(G)$ for which
$\gamma_R(G - E_1) > \gamma_R(G)$.
This number is a measure of the efficiency of Roman domination in graphs. 
In \cite{bhsx}, Bahremandpour et al. showed that
the decision problem for $b_R(G)$ is $NP$-hard even for bipartite graphs. 
For more information we refer the reader to \cite{bhsx, rv0,samcz,samijgta}.

When we remove a vertex from a graph G, the Roman domination number 
can increase or decrease, or remain the same. 
The following classes of graphs were defined and first investigated in \cite{rv} by Jafari Rad and Volkmann. 

\begin{itemize}
\item[$\bullet$] $\mathcal{R}_{CV R}$ is the class of graphs $G$ such that 
                    $\gamma_R(G-v) \not= \gamma_R(G)$ for all $v \in V (G)$,
\item[$\bullet$] $\mathcal{R}_{UV R}$ is the class of graphs $G$ such that 
                    $\gamma_R(G-v) = \gamma_R(G)$ for all $v \in V (G)$.
	\end{itemize}

Here we concentrate on the class $\mathcal{R}_{UV R}$. 
The paper is organized as follows. 
Section 2 contains some known facts  about Roman domination in graphs.
In Section 3 we obtain tight upper bounds for $\gamma_R(G)$ 
and $b_R(G)$ provided a graph $G$ is in $\mathcal{R}_{UVR}$.  
In Section 4 we present necessary and sufficient conditions for 
a tree to be in the class $\mathcal{R}_{UV R}$. We also give 
a constructive characterization of $\mathcal{R}_{UVR}$-trees 
using labellings.

\section{Some  known results}


\begin{lema}\label{on} {\rm(\cite{cd04})} 
Let $f = (V_0; V_1; V_2)$ be any $\gamma_R$-function on a graph $G$. 
Then each component of a  graph $\left\langle V_1 \right\rangle$ has order at least $2$  and 
no edge of $G$ join $V_1$ and $V_2$. 
\end{lema}

Lemma  \ref{on}   will be used in the sequel without specific reference.

\begin{lema}\label{minus} {\rm(\cite{rhv})} 
Let $v$ be a vertex of a graph $G$. Then $\gamma_R(G-v) < \gamma_R(G)$ 
 if and only if there is a $\gamma_R$-function $f = (V_0; V_1; V_2)$ on $G$
 such that $v \in V_1$. If $\gamma_R(G-v) < \gamma_R(G)$ then $\gamma_R(G-v) = \gamma_R(G) - 1$. 
\end{lema}

\begin{lema}\label{minuse} {\rm(\cite{rv})} 
If $e$ is an edge of a graph $G$, then $\gamma_R(G-e) \geq \gamma_R(G)$. 
\end{lema}

\begin{them}\label{r} {\rm(\cite{cd04})} 
A graph G is Roman if and only if it has a $\gamma_R$-function $f$ with $V_1^f = \emptyset$. 
\end{them}

\begin{them}\label{un}{\rm{\cite{fv}}}
Let $T$ be a tree of order at least $3$ and let $D$ be a dominating set of $G$. 
Then $D$ is the unique $\gamma$-set of $T$ if and only if every vertex in 
$D$ has at least two nonadjacent $D$-private neighbors. 
\end{them}

The differential of a graph was introduced in \cite{mhhhs} in 2006, 
 motivated by its applications to information diffusion in social networks. 
The {\it differential of a vertex set} $S\subseteq V(G)$  is defined as
$\partial(S) = |B(S)|-|S|$, where $B(S)$ is the set of vertices in $V(G) - S$ that have a neighbor
in  $S$, and the {\it differential of a graph} $G$ is defined as 
$\partial(G) =\max\{\partial(S) \mid S \subseteq V(G) \}$. 
A set  $S \subseteq V(G)$ is a $\partial$-{\it set} of $G$  if $\partial(S) = \partial(G)$.

\begin{them}\label{diff}\cite{bfs}
Let $G$ be a graph. 
\begin{itemize}
\item[(i)] Then $\gamma_R(G) + \partial(G) = |V(G)|$.
\item[(ii)] An RDF $f = (V_0; V_1; V_2)$ is a $\gamma_R$-function on $G$ 
 if and only if $V_2$ is a $\partial$-set of $G$ and $V_0 = B(V_2)$.
\end{itemize}
\end{them}

\section{Upper bounds}

\begin{observation} \label{disc}
A graph $G$ is in $\mathcal{R}_{UVR}$ if and only if  all its components are also in $\mathcal{R}_{UVR}$.  
\end{observation}

\begin{observation} \label{pn3}
Let a graph $G$ be in $\mathcal{R}_{UVR}$. Then $G$ is a Roman graph. 
If $f = (V_0^f; V_1^f; V_2^f)$ is a $\gamma_R$-function  on $G$
then $V_1^f = \emptyset$, $V_2^f$ is a $\gamma$-set of $G$
 and for any $v \in V_2^f$, $|pn[v,V_2^f]| \geq 3$. 
If $D$ is a $\gamma$-set of $G$ then $h=(V(G)-D;\emptyset;D)$ 
is a $\gamma_R$-function on $G$. 
\end{observation}
\begin{proof}
As $G\in \mathcal{R}_{UVR}$, it follows by Lemma \ref{minus} that
 $V_1^g = \emptyset$ for any $\gamma_R$-function $g$ on $G$. 
Now Theorem \ref{r} implies a graph  G is Roman. 
 Let $f$ be a $\gamma_R$-function  on $G$. 
Since $V_2^f$ is a dominating set of $G$ and $\gamma(G) = \gamma_R(G)/2 = |V_2^f|$, 
$V_2^f$ is a $\gamma$-set of $G$. 
Assume $v \in V_2^f$ and  $|pn[v,V_2^f]| < 3$. 
If $v \not\in pn[v,V_2^f]$ then $l_1 = ((V_0^f - pn[v,V_2^f]) \cup \{v\}; pn[v,V_2^f];  V_2^f - \{v\})$ 
is an RDF on $G$ with weight at most $\gamma_R(G)$ and $V_1^g \not = \emptyset$, a contradiction. 
 If $v \in pn[v,V_2^f]$ then $l_2 = (V_0^f - pn[v,V_2^f]; pn[v,V_2^f];  V_2^f - \{v\})$ 
is an RDF on $G$ with weight not more $\gamma_R(G)$ and $V_1^g \not = \emptyset$, again a contradiction.
Thus, $|pn[v,V_2^f]| \geq 3$. 
 
Finally, the weight of $h$ is $2|D| = 2 \gamma(G) = \gamma_R(G)$ which shows that 
$h$   is a $\gamma_R$-function  on $G$. 
\end{proof}

\begin{remark}\label{e1}
Let  $f$ be a $\gamma_R$-function  on a graph $G \in \mathcal{R}_{UVR}$. 
If $v \in V_2^f$ then  $|pn[v,V_2^f]|$ can be arbitrarily large. 
Indeed, let us consider the graph  $G = (H_1 \cup H_2) +e $,  
where $H_1$ and $H_2$ are disjoint copies of $K_r$, $r \geq 4$.   
Clearly $\gamma_R(G) = 4$,  $G \in \mathcal{R}_{UVR}$ and 
if $x_i \in V(H_i)$, $i=1,2$, then  $f = (V(G) - \{x_1,x_2\}; \emptyset; \{x_1,x_2\})$ 
is a $\gamma_R$-function on $G$ and  $|pn[x_i,V_2^f]| \in \{r-1, r\}$.  
\end{remark}

It is easy to see that the following grahs are in $\mathcal{R}_{UV R}$: 
(a) $K_n$, $n \geq 3$; (b)  $K_{m,n}$ for $m \geq n \geq 4$, (c) $P_{3k}$ and $C_{3k}$, $k \geq 1$; (d) the cube and icosahedron.

Chambers et al. \cite{ck10} proved that if  $G$ is a graph with $\delta(G) \geq 1$  then 
$\gamma_R(G) \leq 4n/5$. For the graphs in $\mathcal{R}_{UVR}$ this bound can be lowered. 

\begin{proposition}\label{3v2}
Let $G \in \mathcal{R}_{UVR}$ be a connected graph of order $n$. 
 Then $\frac{2}{3}n \geq \gamma_R(G)$. 
If the equality holds then for any $\gamma_R$-function  $f$ on $G$, 
$V_2^f$ is an efficient dominating set of $G$ and each vertex of $V_2^f$  has degree $2$. 
If $G$ has an efficient dominating set $D$   and each vertex of $D$  has degree $2$ 
then $\frac{2}{3}n = \gamma_R(G)$. 
\end{proposition}
\begin{proof}
Let   $f$ be any $\gamma_R$-function  on $G$. By  Observation \ref{pn3}, 
 $V_1^f = \emptyset$ and  $|pn[v,V_2^f]| \geq 3$ when $v \in V_2^f$. 
Hence 
\begin{equation}\label{eq1}
|V_0^f| = |\cup_{v \in V_2^f}(N(v)-V_2^f)|  \geq \Sigma_{v\in V_2^f} (|pn[v,V_2^f]|-1) \geq 2|V_2^f|  = \gamma_R(G).
\end{equation}
\indent  Therefore,  $n = |V_0^f| + |V_2^f| \geq \frac{3}{2}\gamma_R(G)$. 

Suppose $n = \frac{3}{2}\gamma_R(G)$.  Then all the above inequalities must be equalities.
If equality holds on the left side of \eqref{eq1} then  $N[v] =  pn[v,V_2^f]$ for all $v \in V_2^f$,  
which implies $V_2^f$ is an efficient dominating set in $G$. 
If in addition,  the right side of \eqref{eq1}  becomes equality then $|N(v) |=2$ for each $v \in V_2^f$. 

Assume now that $D$ is   an efficient dominating set of $G$ 
and all vertices of $D$ have degree $2$.  Hence  $n = 3|D|$.  
Since $G$ is Roman and each efficient dominating set is a $\gamma$-set 
$\gamma_R(G) = 2\gamma(G) = 2|D| = \frac{2}{3}n$ as required. 
\end{proof}

The bound  in Proposition \ref{3v2} is tight at least for all cycles $C_{3k}$, $k \geq 1$. 
 In the next section we present a constructive characterization of all trees $T$
  with $|V(T)| = \frac{3}{2}\gamma_R(T)$.

 Jafari Rad and Volkmann in \cite{rv0}  proved that
 $b_R(G) \leq deg(x) + deg(y) + deg(z) - |N(x) \cap N(y)| - 3$ for any path $P: x,y,z$ in a graph $G$. 
 For all graphs $G$ belonging to $\mathcal{R}_{UVR}$, this bound can be improved to $\delta(G)$.

\begin{proposition}\label{02}
Let $G$ be a graph and $v \in V(G)$. 
If for any $\gamma_R$-function $f$ on $G$, $f(v) \not = 1$ then $\gamma_R (G-E_v) > \gamma_R (G)$, 
where $E_v$ is the set of all edges incident to a vertex $v$. 
In particular, $b_R(G) \leq deg(v,G) \leq \Delta(G)$. 
\end{proposition}
\begin{proof}
By Lemma \ref{minuse}, $\gamma_R (G-E_v) \geq \gamma_R (G)$. 
Consider any $\gamma_R$-function $g$ on $G-E_v$. 
Clearly $g$ is an RDF on $G$. Since $v$ is an isolated vertex in $G-E_v$, 
 $g(v)=1$. But then  $g$ is no  $\gamma_R$-function on $G$. 
Thus $\gamma_R (G-E_v) > \gamma_R (G)$ which implies $b_R(G) \leq deg(v,G) \leq \Delta(G)$. 
 \end{proof}

\begin{corollary}\label{uvrbon}
If a graph $G$ is in $\mathcal{R}_{UVR}$  then $b_R(G) \leq \delta(G)$.
\end{corollary}
\begin{proof}
By applying Proposition \ref{02} to the graph $G$ and a vertex $v$ of minimum degree we obtain the result. 
\end{proof}

The bound stated in Corollary \ref{uvrbon} is tight. For example  when 
(a) $G = C_{3k}$, $k \geq 1$, and (b) $\delta(G) = 1$. As an immediate consequence we obtain:

\begin{corollary}\label{uvrtree}
For any tree $T$ in $\mathcal{R}_{UVR}$, $b_R(T) =1$.
\end{corollary}

Note that for a tree $T$ of order at least three  Ebadi and PushpaLatha \cite{epl}, 
and Jafari Rad and Volkmann \cite{rv0},  independently proved that $b_R(T) \leq 3$.

\section{Small number of edges}

In this section we give necessary and sufficient conditions for a tree to be in $\mathcal{R}_{UVR}$.
In particular we present here a constructive characterization of $\mathcal{R}_{UVR}$-trees 
using labellings. We define a labeling of a tree $T$ as a function $S:V (T ) \rightarrow\{A,B,C\}$. 
 The label of a vertex $v$ is also called its status, denoted $sta_T(v)$. 
A labeled tree is denoted by a pair $(T , S)$. We denote the sets of vertices of status $A$, $B$ and $C$ 
 by $S_A(T )$, $S_B(T)$ and $S_C(T )$, respectively, or simply by $S_A$, $S_B$ and $S_C$ 
if the tree $T$ is clear from context. 
By a labeled $K_{1,2}$ we shall mean a copy of $K_{1,2}$ whose leaves have status 
 $A$ and the status of  the central vertex is $B$.

 Let $\mathscr{T}$ be the family of labeled trees $(T,S)$ that can be obtained from a sequence 
$(T_1,S_1), \dots, (T_j,S_j)$, ($j \geq 1$), of labeled trees such that
 $(T_1,S_1)$ is a labeled $K_{1,2}$ and $(T,S) = (T_j,S_j)$, 
and, if $j \geq 2$,  $(T_{i+1},S_{i+1})$ can be obtained recursively from $(T_i,S_i)$ 
 by one of the four operations $O1$, $O2$, $O3$ and $O4$ listed below.

{\bf Operation}  {\em O1}.  The labeled tree $(T_{i+1},S_{i+1})$ 
is obtained from $(T_i,S_i)$ by adding a path $x,y,z$ and the
edge $ux$ where $u \in V(T_i)$ and $sta(u) \in \{A, C\}$, 
and letting $sta(x)=sta(z) = A$  and $sta(y)= B$.

{\bf Operation}  {\em O2}.   The labeled tree $(T_{i+1},S_{i+1})$ 
 is obtained from $(T_i,S_i)$ by adding a star with  leaves $x,z,t$ 
and a central vertex $y$, and then adding the edge $ux$ where $u \in V(T_i)$
 and $sta(u)=B$, and letting $sta(x)=C$, $sta(z) = sta(t) = A$  and $sta(y)= B$.

\begin{figure}[htbp]
	\centering
		\includegraphics{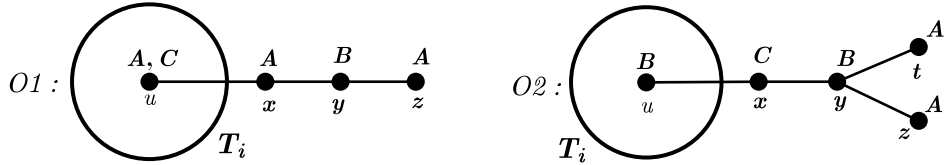}
	\caption{Operations $O1$ and $O2$}
	\label{fig:O12}
\end{figure}

{\bf Operation}  {\em O3}.  The  labeled tree $(T_{i+1},S_{i+1})$
 is obtained from $(T_i,S_i)$ by adding a path $x,y,z$ and the
edge $uy$ where $u \in V(T_i)$ and $sta(u) =C$, and letting $sta(x)=sta(z) = A$  and $sta(y)= B$.

By a labeled tree $R$ we shall mean a labeled tree obtained from  a labeled $K_{1,2}$ by operation $O2$.

{\bf Operation}  {\em O4}. The labeled tree $(T_{i+1},S_{i+1})$ 
is obtained from $(T_i,S_i)$ by adding a labeled $R$ and the
edge $ux$ where $u \in V(T_i)$ and $sta(u) \in \{A,C\}$, and $x \in V(R)$ with  $sta(x)=C$.

\begin{figure}[htbp]
	\centering
		\includegraphics{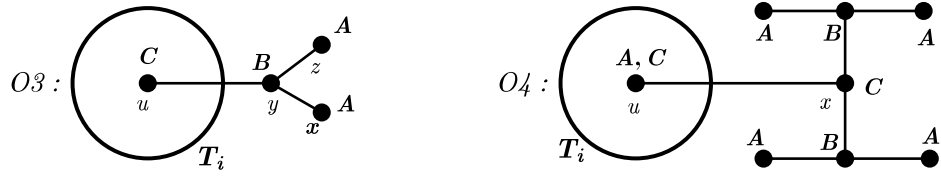}
	\caption{Operations $O3$ and $O4$}
	\label{fig:O34}
\end{figure}

Remark that once a vertex is assigned a status, 
this status remains unchanged as the labeled tree $(T,S)$ is recursively constructed. 

\begin{observation} \label{sabc}
Let $(T,S)$ be in $\mathscr{T}$. Then 
\begin{itemize}
\item[(i)] $S_B$ is an independent dominating set and for each vertex $v$ in $S_B$, 
 $|N(v) \cap S_A| =2$ and $pn[v,S_B] = (N(v) \cap S_A) \cup\{v\}$. 
\item[(ii)] Each vertex in $S_A$ is adjacent to exactly one vertex in $S_B$ and $|S_A| = 2|S_B|$. 
\item[(iii)] Each vertex in $S_C$ is adjacent to at least $2$ vertices in $S_B$.
\item[(iv)] $S_B$ is the unique $\gamma$-set of $T$.
\end{itemize}
\end{observation}
\begin{proof}
(i)--(iii) By the definition of $(T,S)$.

(iv) Theorem \ref{un} and (i) together imply the result.
\end{proof}

\begin{corollary}\label{unilab}
If $(T,S_1),(T,S_2)\in \mathscr{T}$ then $S_1 \equiv S_2$.
\end{corollary}

\begin{proof}
Immediately by Observation \ref{sabc}. 
\end{proof}

Let $(T,S)\in \mathscr{T}$. By the above corollary, $S$ is unique. 
So, when the context is clear we shall write $T \in \mathscr{T}$ instead of $(T,S)\in \mathscr{T}$.
 
We define the following classes of graphs: 
\begin{itemize}
 \item[$\bullet$] $\partial_{CV R}$ is the class of graphs $G$ such that 
                    $\partial(G-v) \not= \partial(G)$ for all $v \in V (G)$, and
	\item[$\bullet$] $\partial_{UV R}$ is the class of graphs $G$ such that 
                    $\partial(G-v) = \partial(G)$ for all $v \in V (G)$.
\end{itemize}

By Theorem \ref{diff} it immediately follows the next observation. 
\begin{observation}\label{equi}
$\mathcal{R}_{UVR} = \partial_{UVR}$ and $\mathcal{R}_{CVR} = \partial_{CVR}$.
\end{observation}

\begin{theorem}\label{main}
For any tree $T$ of order at least three the following assertions  are equivalent.
\begin{itemize}
\item[(i)] $T$ is in  $\mathscr{T}$.
\item[(ii)] $T$ is in $\mathcal{R}_{UVR}$.
\item[(iii)] $T$ has a unique $\gamma_R$-function, say $f$,  and  all the following holds: 
                   $V_1^f = \emptyset$, $V_2^f$ is independent and $|pn[v,V_2^f]| = 3$ for each $v \in V_2^f$.
\item[(iv)] $T$ has a unique $\gamma$-set $D$, $D$ is independent and $|pn[v,D]| = 3$ for each $v \in D$.
\item[(v)] $T$ is in $\partial_{UVR}$.
\end{itemize}
\end{theorem}
\begin{proof}
For any labeled tree $(H,S) \in \mathscr{T}$, $S_B(H)$ is a dominating set of $T$ 
(by Observation \ref{sabc}) and hence $f_H = (S_A(H) \cup S_C(H); \emptyset; S_B(H))$ is an RDF on $H$.

{\bf Claim 1}:  If $f_H$  is the unique $\gamma_R$-function on $(H,S) \in  \mathscr{T}$ then $H$ is in $\mathcal{R}_{UVR}$.

{\em Proof.}  Since $f_H$  is the unique $\gamma_R$-function on $(H,S) \in  \mathscr{T}$ and $V_1^{f_H}$ is empty, 
Lemma \ref{minus} implies that $\gamma_R(H-x) \geq \gamma_R(H)$ for each $x \in V(H)$. 
	If $x \in S_A(H) \cup S_C(H)$ then $f_H|_{H-x}$ is an RDF on $H-x$ of weight  $ \gamma_R(H)$. 
	Now, let $x \in S_B(H)$. Then $x$ has exactly $2$ private neighbors with respect to $S_B(H)$ (by Observation \ref{sabc}) , 
	say $y$ and $z$. 		Define $f_H^x : \{0,1,2\} \rightarrow V(H-x)$ as $f_H^x(y) = f_H^x(z) =1$
	and $f_H^x(t) = f_H(t)$ otherwise.  Since the weight of $f_H^x$ is $\gamma_R(H)$, we obtain $H \in \mathcal{R}_{UVR}$. 
					
			(i) $\Rightarrow$ (ii): Let $(T,S)$ be in  $\mathscr{T}$. By Claim 1, it is sufficient to prove that 
			                                           $f_T$ is actually  the unique $\gamma_R$-function on $T$. 
						We now proceed by induction on $|S_B|$. The base case is immediate since $T$ is a labeled star $K_{1,2}$. 
			Let $k \geq 2$ and suppose that for all labeled trees $(H, S^\prime) \in \mathscr{T}$ with $|S_B^\prime(H)| < k$ that 
			$f_H = (S_A^\prime(H) \cup S_C^\prime(H);\emptyset; S_B^\prime(H))$ is the unique $\gamma_R$-function on $H$. 
		By Claim 1, $H \in \mathcal{R}_{UVR}$. 
									
			Let $(T,S) \in \mathscr{T}$ have $|S_B(T)| =k$. Then $T$ can be obtained from a sequence
			$T_1=K_{1,2}, T_2,..,T_k =T$ of labeled trees, and  $T_{i+1}$ can be obtained from $T_i$ 
			by operation $O1$, $O2$, $O3$ or $O4$ for $i = 1,..,k-1$. All $T_i$ are clearly in $\mathscr{T}$. 
			We consider four possibilities depending on whether $T$ is obtained from $U=T_{k-1}$ 
			by operation $O1$, $O2$, $O3$ or $O4$. Note that $U \in \mathcal{R}_{UVR}$. 
			
{\em Case} 1: 	$T$ is obtained from $U=T_{k-1}$ by operation 	$O1$. 
                           	Suppose $T$ is obtained from $U$ by adding a path  $x,y,z$ and the edge $ux$
														where $u \in V(U)$ and $sta(u) \in \{A, C\}$,  $sta(x)=sta(z) = A$  and $sta(y)= B$.
			                      Clearly $f_T|_U = f_U$ which leads to $\gamma_R(U) = f_U(V(U)) = f_T(V(T)) -2 \geq \gamma_R(T) - 2$. 
																											
														Now let $f$ be any $\gamma_R$-function on $T$. 
														Suppose $f(u) \geq 1$. Then the weight of $f|_{U}$ would be  greater than $\gamma_R(U)$ and 
														$f(x) + f(y) + f(z) =2$, which leads to $\gamma_R(T) > 2 + \gamma_R(U)$, a contradiction. So, $f(u) = 0$ for each 
														$\gamma_R$-function $f$ on $T$. 
														
															Since $U \in \mathcal{R}_{UVR}$, it follows that $\gamma_R(U-u) = \gamma_R(U)$.  
														Hence if $f(x) \geq 1$ then $f(x) + f(y) + f(z) >3$ and this implies 
														$f(V(T)) > \gamma_R(U) + 2$, a contradiction. Thus $f(x) = 0$ and then $f(y) = 2$, $f(z) = 0$ and  $f|_U=f_U$. 
														All this implies that $f \equiv f_T$ is the unique $\gamma_R$-function on $T$ and 
														$\gamma_R(T) = \gamma_R(U) +2$. 	
														
			{\em Case} 2: 	$T$ is obtained from $U=T_{k-1}$ by operation 	$O2$. 
                           	Suppose $T$ is obtained from $U$ by adding a star $K_{1,3}$ with leaves   $x,z,t$ and a central vertex $y$,  
														and also adding the  edge $ux$ where $u \in V(U)$,  $sta(u) = sta(y) = B$,  $sta(x)=C$ and  $sta(z) = sta(t) = A$. 
															Since obviously  $f_T|_U = f_U$,  $\gamma_R(U) = f_U(V(U)) = f_T(V(T)) -2 \geq \gamma_R(T) - 2$. 
			                  														
														Let $f$ be an arbitrary  $\gamma_R$-function on $T$. 
														Hence either $f(y)=0$ and $f(z) = f(t) = 1$, or $f(y)=2$ and $f(z) = f(t) = 0$.
			                      In the former case we have $f(x)=2$, which leads to $f(u) = 0$. But then 
														since $U \in \mathcal{R}_{UVR}$, $\gamma_R(U) + 4 \leq f(V(T)) = \gamma_R(T)$, a contradiction. 
			                       Hence 	$f(y)=2$ and $f(z) = f(t) = 0$. But then $f(x) =0$ and 	$f|_U=f_U$.  
														From the above we conclude  that $f \equiv f_T$ is the unique $\gamma_R$-function on $T$ and 
														$\gamma_R(T) = \gamma_R(U) +2$. 										
			
			{\em Case} 3: 	$T$ is obtained from $U=T_{k-1}$ by operation 	$O3$. 
                           	Suppose $T$ is obtained from $U$ by adding a path  $x,y,z$  and the edge $uy$ 
														where $u \in V(T_i)$, $sta(u) =C$, $sta(x)=sta(z) = A$  and $sta(y)= B$.
														Since $f_T|_U = f_U$, we obtain $\gamma_R(U) = f_U(V(U)) = f_T(V(T)) -2 \geq \gamma_R(T) - 2$. 
																											
														Now we shall prove that $f_U|_{U-u}$ is the unique $\gamma_R$-function on $U-u$. 
														Since $\gamma_R(U-u) = \gamma_R(U)$ and $u$ has at least two  neigbors in $S_B(U) = V_2^{f_U}$,
														the restriction of $f_U$ on any component of $U-u$ is a $\gamma_R$-function. 
														Suppose there is a $\gamma_R$-function $g$ on $U-u$ different from  $f_U|_{U-u}$. 
														Then there is at least one component of $U-u$, say $U_r$, such that $f_U|_{U_r} \not\equiv g_{U_r}$. 
														Define now an RDF $h$ on $U$ as follows: $h(x) = g|_{U_r}(x)$ when $x \in V(U_r)$ and 
														$h(x) = f_U(x)$ otherwise. But then $h$ and $f_U$ have the same weight, a contradiction. 
														Thus, $f_U|_{U-u}$ is the unique $\gamma_R$-function on $U-u$.

														Now let $f$ be any $\gamma_R$-function on $T$. Obviously $f(y) \not = 1$. 
			                      Suppose  $f(y)=0$. Then $f(u) =2$, $f(x) = f(z) = 1$ and $f|_U$  is an RDF on $U$.  
														Since $f_U$ is the unique $\gamma_R$-function of $U$ and $f_U(u) = 0$, 
														an RDF $f|_U$  has weight more than $\gamma_R(U)$ and then 
														$\gamma_R(T) = f(V(T)) > \gamma_R(U) + 2$, a contradiction.
														Thus $f(y)=2$ and then $f(x) = f(z) = 0$. If $f(u) = 2$ then as above we again obtain a contradiction. 
														So $f(u) =0$  and since $f(y)=2$, it follows that $f|_{U-u}$ is a $\gamma_R$-function on $U-u$. 
														But we already know that $f_U|_{U-u}$ is the unique $\gamma_R$-function on $U-u$. 
														Therefore $f \equiv f_T$. Since $f$ was chosen arbitrarily, $f_T$ is the unique  
														$\gamma_R$-function on $T$ and $\gamma_R(T) = \gamma_R(U) + 2$.

			{\em Case} 4: 		$T$ is obtained from $U=T_{k-1}$ by operation 	$O4$. 								
				                   In this case $T = U \cup R + ux$, where $u \in V(U)$ with $sta(u) \in \{A,C\}$ and 
														$x$ is a central vertex of $R$, and $sta(x) = C$.  Note that $R$ is in $\mathcal{R}_{UVR}$ and 
														$f_R = (S_A(R)\cup S_C(R); \emptyset; S_B(R))$ is the unique $\gamma_R$-function on $R$. 
														Since $f_T|_U \equiv f_U$ and $f_T|_R \equiv f_R$ it follows that 
														$\gamma_R (U) = f_U(V(U)) = f_T(V(T)) - f_R(V(R)) \geq \gamma_R(T) - 4$. 
														Consider any $\gamma_R$-function $f$ on   $T$ and suppose $f(u) \not = 0$. 
														Then the weight of $f|_U$ is more than $\gamma_R(U)$ and the weight of $f|_R$ is $4$. 
														This leads to $\gamma_R(U) > \gamma_R(T) -4$, a contradiction. Hence $f(u) = 0$. 
														Assume now that $f(x) \not = 0$. 
														Since $U$ is in $\mathcal{R}_{UVR}$, $f|_U(V(U)) \geq \gamma_R(U)$ and $f|_R(V(R)) \geq 4+f(x) \geq 5$. 
														Hence $f(V(T)) \geq \gamma_R(U) + 5$, a contradiction. Thus $f(x) =0$. 
														But then $f|_U \equiv f_U$ and $f|_R \equiv f_R$. Thus $f \equiv f_T$ is the unique $\gamma_R$-function on $T$ 
														and $\gamma_R(T) =  \gamma_R(U) +4$.

(ii) $\Rightarrow$ (iii):   
 Let a tree $T$ be in $\mathcal{R}_{UVR}$ and let  $v \in V_2^f$ for some $\gamma_R$-function $f$ on $T$.

{\bf Claim 2.}     Let $u$ be a neighbor of $v$ and $T_u$ be the component of $T-v$ that contains $u$. Then

(a) $f|_{T_u}(V(T_u)) \leq \gamma_R(T_u) \leq f|_{T_u}(V(T_u)) + 1$.

(b) There are exactly $2$ components of $T-v$, say $Q_1$ and $Q_2$, such that $\gamma_R(Q_i) = f|_{T_u}(V(Q_i)) + 1$, $i=1,2$.

(c) If $u_i \in V(Q_i)$ is a neighbor of $v$, $i=1,2$, then $pn[v,V_2^f] = \{u_1, u_2, v\}$.

(d) $V_2^f$ is an independent set.

{\em Proof.}  
 (a) Assume $\gamma_R(T_u) <  f|_{T_u}(V(T_u))$ and let $g$ be any $\gamma_R$-function on $T_u$. 
        But then the function $h$ defined by $h(x) = f(x)$ when $x \in V(T) - V(T_u)$ and $h(x) = g(x)$ when $x \in V(T_u)$, 
				is an RDF on $T$ with weight less than $f(V(T))$, which is impossible. 
				
				Now assume $\gamma_R(T_u) >  f|_{T_u}(V(T_u))$. Hence  				
				$f|_{T_u}$ is no RDF on $T_u$. Then $f(u)=0$ and $V_2^{f|T_u}$ dominated $T_u-u$. 
				Define an RDF $l$ on $T_u$ by $l(u) =1$ and $l(x) = f(x)$ for all $x \in V(T_u-u)$. 
				Since $l(V(T_u)) = f|_{T_u}(V(T_u)) + 1$, the right side inequality is true. 
				
(b) Since $\gamma_R(T-v) = \gamma_R(T)$ and $f(v) =2$, the result follows by (a). 				

(c) By the proof of this claim up to here we know that $V_2^f - \{v\}$ dominates $N(v) - \{u_1,u_2\}$ and 
      neither $u_1$ nor $u_2$ is dominated by $V_2^f - \{v\}$. Thus $u_1,u_2 \in pn[v,V_2^f]$. 
			If $v \not\in pn[v,V_2^f]$ then the RDF  $g$ on $T$ defined by $g(x)=f(x)$ for all $x \in V(T) - \{u_1, u_2, v\}$, 
			$g(v)=0$ and $g(u_1) = g(u_2) =1$ is a $\gamma_R$-function on $T$ with $V_1^g \not = \emptyset$, a contradiction.
			
(d) The result immediately follows by (c).

We are now ready to prove the uniqueness of $f$. 			
			 Suppose there is a vertex $u \in N(v)$ such that $g(u) =2$ for some $\gamma_R$-function $g$ on $T$. 
		By Claim 2, $g \not \equiv f$ and $g(v) =0$.  Let without loss of generality, $u \not \equiv u_1$. 
		Hence $g|_{Q_1}$ is a $\gamma_R$-function on $T_1$. 
		By the proof of Claim 2 we already know that there is a $\gamma_R$-function $l$ on $Q_1$ with $l(u_1) =1$. 
		Consider now the $\gamma_R$-function $g_1$ on $T$ defined by $g_1(x) = l(x)$ when $x \in V(Q_1)$ and 
		$g_1(x) = g(x)$ otherwise.  Since $g_1(u_1) = l(u_1) =1$, $V_1^{g_1}$ is not empty, a contradiction.

(iii) $\Rightarrow$ (iv): Since $V_2^f$ is a dominating set of $T$, by Theorem \ref{un} it follows that 
                                              $V_2^f$ is the unique $\gamma$-set of $T$.

(iv) $\Rightarrow$ (i):  Denote by $\mathscr{H}$ the set of all trees $T$ for which (iv) holds. 
We shall prove that if $T \in \mathscr{H}$ then $T \in \mathscr{T}$. 
We proceed by induction on the domination number of $T$. 
If $T \in \mathscr{H}$ and $\gamma(T) =1$ then $T \equiv K_{1,2}$ and we are done. 
So, let $T \in \mathscr{H}$, $\gamma(T)=k \geq 2$ and suppose that each tree $H \in \mathscr{H}$ with 
$\gamma(H)  <k$ is in $\mathscr{T}$. Let $P: x_1,x_2,..,x_n$ be any diametral path in $T$. 
Then $x_n$ is a leaf and $x_{n-1} \in D$, where $D$ is the unique $\gamma$-set of $T$. 

{\it Case} 1: $deg(x_{n-1}) = 2$. 
Since  $T \in \mathscr{H}$,$\{x_{n-2},x_{n-1}, x_n\} = pn[x_{n-1},D]$ and all 
neighbors of $x_{n-2}$ but $x_{n-1}$ are in $V(T)-D$. 
Now by the choice of $P$, $N(x_{n-2}) = \{x_{n-1}, x_{n-3}\}$. But then $T-x_{n-3}x_{n-2}$ has exactly $2$ components, 
say $F_1$ and $F_2$, where $F_1 \equiv K_{1,2}$ is induced by $\{x_{n-2}, x_{n-1},x_n\}$. 
Since the set $D_1 = D-\{x_{n-1}\}$ is an independent dominating set of $F_2$ and $|pn[v,D_1]| = 3$ 
for each $v \in D_1$, Theorem \ref{un} implies that $D_1$ is the unique $\gamma$-set of $F_2$. 
Hence $F_2 \in \mathscr{H}$. By inductive hypothesis, $F_2 \in \mathscr{T}$.  
Since $x_{n-3} \not\in D_1$, $sta_{F_2}(x_{n-3})\in \{A,C\}$ (by Observation \ref{sabc}).  
 Let us consider $F_1$ as a labeled $K_{1,2}$.  But then $T$ is obtained from $F_2$ by operation $O1$. 
Thus, $T \in \mathscr{T}$.

{\it Case} 2: $deg(x_{n-1}) \geq 3$. 
By the choice of $P$, $x_{n-2}$ is the unique non-leaf neighbor of $x_{n-1}$. 
Now $deg(x_{n-1}) \geq 3$,  $|pn[x_{n-1},D]| = 3$ and $D$ is independent    together imply 
(a) $x_{n-1}$ is adjacent to exactly $2$ leaves, $x_n$ and say $y$, and (b) $|pn[x_{n-1},D]| = \{x_{n-1},x_n,y\}$. 
 First suppose $x_{n-2}$ is adjacent to at least $3$ vertices in $D$. 
 Then $T-x_{n-2}x_{n-1}$ has exactly $2$ components, say $F_3$ and $F_4$, 
 where $F_3 \equiv K_{1,2}$ with  $V(F_3) = \{x_{n-1},x_n,y\}$. 
 Since the set $D_2 = D-\{x_{n-1}\}$ is an independent dominating set of $F_4$ and $|pn[v,D_2]| = 3$ 
for each $v \in D_2$, Theorem \ref{un} implies that $D_2$ is the unique $\gamma$-set of $F_4$. 
Hence $F_4 \in \mathscr{H}$. By inductive hypothesis, $F_4 \in \mathscr{T}$ and by Observation \ref{sabc},
 $D_2$ consists of all vertices having status $B$.  
Since $x_{n-2} \not\in D_2$ and $x_{n-2}$ is adjacent to at least $2$ elements of $D_2$, 
again by Observation \ref{sabc} it follows $x_{n-2}$ has status $C$. 
 Let us consider $F_3$ as a member of $\mathscr{T}$. 
 But then $T$ is obtained from $F_4$ by operation $O3$. 
Thus, $T \in \mathscr{T}$.

So, let $z$ and $x_{n-1}$ are all neighbors of $x_{n-2}$ in $D$. 
Suppose first that $N(x_{n-2}) = \{x_{n-1},z\}$. 
Then $T-x_{n-2}z$ has exactly $2$ components, say $F_5$ and $F_6$, 
 where  $V(F_5) = \{x_{n-2}, x_{n-1}, x_n, y\}$. 
 Since the set $D_3 = D-\{x_{n-1}\}$ is an independent dominating set of $F_6$ and $|pn[v,D_3]| = 3$ 
for each $v \in D_3$, Theorem \ref{un} implies that $D_3$ is the unique $\gamma$-set of $F_6$. 
Hence $F_6 \in \mathscr{H}$. By inductive hypothesis, $F_6 \in \mathscr{T}$. 
Define labeling $S: V(T) \Rightarrow \{A,B,C\}$ as follows: (a) the restriction of $S$ on $F_6$ coincide 
with the  unique labeling of $F_6$ as a member of $\mathscr{T}$, and  (b) $S(x_{n-2}) = C$, $S(x_{n-1}) = B$ and 
$S(x_n) = S(y) = A$. Since $z \in D_3$,  $S(z) =B$ (by Observation \ref{sabc}). 
 But then $T$ is obtained from $F_6$ by operation $O2$. 
Thus, $T \in \mathscr{T}$.  
Finally let  $x_{n-2}$ have neighbors in $V(T)-D$. 
By the choice of $P$, (a) $x_{n-2}$ has exactly one neighbor in $V(T)-D$, say $u$, and 
(b) $z$ is a support vertex of degree $3$ which has $2$ leaves as neighbors.  
Then $T-x_{n-2}u$ has exactly $2$ components, say $F_7$ and $F_8$,  where  $u \in V(F_8)$ and $F_7$ is an unlabeled $R$. 
 Since the set $D_5 = D-\{x_{n-1},z\}$ is an independent dominating set of $F_8$ and $|pn[v,D_5]| = 3$ 
for each $v \in D_5$, Theorem \ref{un} implies that $D_5$ is the unique $\gamma$-set of $F_8$. 
Hence $F_8 \in \mathscr{H}$. By inductive hypothesis, $F_8 \in \mathscr{T}$.

Define labeling $S^{\prime}: V(T) \Rightarrow \{A,B,C\}$ as follows:
 (a) the restriction of $S^{\prime}$ on $F_8$ coincide  with the  unique labeling of $F_8$ as a member of $\mathscr{T}$, 
 and  (b)  $F_7$ together with  the restriction of $S^{\prime}$ on $F_7$  form a labelled $R$.  
 But then $T$ is obtained from $F_8$ by operation $O4$. 
Thus, $T \in \mathscr{T}$.   

(ii) $\Leftrightarrow$ (v): Immediately by Observation \ref{equi}.
{\tiny{$\blacksquare$}}	
\end{proof}

By the  proof of the previous theorem it immediately follows:
\begin{corollary}\label{SB}
If    $(T,S) \in  \mathscr{T}$ then $f_T = (S_A(T) \cup S_C(T); \emptyset; S_B(T))$ 
is the unique $\gamma_R$-function on $T$.  
\end{corollary}

The class $URD$ of all graphs which have exactly one  $\gamma_R$-function 
was introduced and  investigated by Chellali and Rad in \cite{cr}. 
Theorem \ref{main} shows that  any  tree in $\mathcal{R}_{UVR}$ is also  in $URD$.

\begin{corollary}\label{vdel}
Let  $f$ be the unique $\gamma_R$-function on a tree $T \in \mathcal{R}_{UVR}$. 
If $u,v  \in pn[x,V_2^f]$ for some $x \in V_2^f$ then 
 $\gamma_R(T-\{u,v\}) = \gamma_R(T)-1$. 
\end{corollary}
\begin{proof}
By Theorem \ref{main}, $|pn[x,V_2^f]| = 3$ and $x \in pn[x,V_2^f]$. 
Define an RDF $g$ on $T-u$ by $g=(V_0^f-pn[x,V_2^f]; pn[x,V_2^f]-\{u\}; V_2^f-\{x\})$. 
Since $f$ and $g$ have the same weights, $g$ is a  $\gamma_R$-function on $T-u$. 
Now, $v \in V_1^g $ and  Lemma \ref{minus} lead to $\gamma_R(T-\{u,v\}) = \gamma_R(T)-1$. 
\end{proof}

\begin{corollary}\label{edel}
Let $f$ be the unique $\gamma_R$-function on a tree $T \in \mathcal{R}_{UVR}$. 
If $x,y \in V_0^f$ and $xy \in E(G)$ then $T-xy$ and all its components are in $\mathcal{R}_{UVR}$.
\end{corollary}
\begin{proof}
Since  $x,y \in V_0^f$, $f_1 = f|_{T-xy}$ is an RDF on $T-xy$. Hence $\gamma_R(T-xy) \leq \gamma_R(T)$. 
Now by Lemma \ref{minuse} it follows that  $\gamma_R(T-xy) =\gamma_R(T)$. 
Therefore $f_1$ is a $\gamma_R$-function on $T-xy$ and any  $\gamma_R$-function on $T-xy$ 
is a $\gamma_R$-function on $T$. By the uniqueness of $f$ on $T$ it follows that $f_1$ is 
 the unique $\gamma_R$-function on $T-xy$. Since $V_i^{f_1} = V_i^f$ for $i = 1,2,3$, 
the statement (iii) of Theorem \ref{main} holds for any component $U$ of  $T-xy$ and $f_1|_U$. 
Thus $U$ is in $\mathcal{R}_{UVR}$ because of   Theorem \ref{main}. 
But then $T-xy$ is also in $\mathcal{R}_{UVR}$ (by Observation \ref{disc}).
\end{proof}

Recall that for any graph $G \in \mathcal{R}_{UVR}$,  $\frac{2}{3}|V(G)| \geq \gamma_R(G)$ 
(Proposition \ref{3v2}).
Now we characterize all trees $T$ for which  $\frac{2}{3}|V(T)| = \gamma_R(T)$. 
Define $\mathscr{T}_1 = \{(T,S) \in \mathscr{T} \mid S_C(T) = \emptyset\}$. 
Clearly $(T,S) \in \mathscr{T}_1$ if and only if 
it is sufficient  to use only the operation $O1$ for building   $(T,S)$ 
from a labeled $K_{1,2}$  

\begin{proposition}\label{=3/2}
 If a tree  $T \in \mathcal{R}_{UVR}$ then   $\frac{2}{3}|V(T)| = \gamma_R(T)$ if and only if $T \in \mathscr{T}_1$.
\end{proposition}
\begin{proof}
By Theorem \ref{main} and Corollary \ref{SB} it follows that 
$(T,S) \in  \mathscr{T}$ and $f_T = (S_A(T) \cup S_C(T); \emptyset; S_B(T))$ is 
 the unique $\gamma_R$-function on $T$. 
Then $\gamma_R(T) = 2|S_B(T)|$ and $|S_A(T)| = 2|S_B(T)|$ (by  Observation \ref{sabc}).  
Now $|V(G)| = |S_A(T)| + |S_B(T)| + |S_C(T)|  \geq |S_A(T)| + |S_B(T)| = \frac{3}{2}\gamma_R(T)$ 
with equality if and only if $S_C(T) = \emptyset$, as required.
\end{proof}

\begin{theorem}\label{minedges}
Let a connected $n$-order graph $G$ be in $\mathcal{R}_{UVR}$. 
Let the size of $G$ be minimum. 
\begin{itemize}
\item[(i)] Then  $|E(G)| = n-1$ if and only if   $n \in \{3,6,7\} \cup \{9, 10, ...\}$  and $G \in \mathscr{T}$.
\item[(ii)] If $n \in \{4,5\}$ then $|E(G)| = 2n-3$ and $G = K_2 \vee \overline{K_{n-2}}$.  
\item[(iii)] If $n=8$ then $|E(G)|	= 8$ and $G$  is the graph depicted in Figure \ref{fig:UVR}.							
\end{itemize}
\end{theorem}

\begin{figure}[h]
	\centering
		\includegraphics{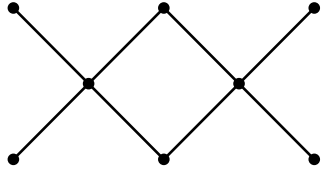}
	\caption{The unique 8-vertex unicyclic graph in  $\mathcal{R}_{UVR}$. }
	\label{fig:UVR} 
\end{figure}

\begin{proof}
(i) If $G$ is a tree then by Theorem \ref{main}, $G$ is in $\mathscr{T}$. 
     	Now either $G$ is a labeled $K_{1,2}$ or  there is a labeled tree $U \in \mathscr{T}$ 
		such that $G$ is obtained from $U$  by applying one of   operations $O1$, $O2$, $O3$ and $O4$ once.
		Hence the order of $G$ is $|V(U)| +3$ or  $|V(U)| +4$ or $|V(U)| +7$. 
 This immediately implies the desired result.

(ii) 	By checking all connected graphs of order $4$ and $5$ we obtain the result 
                        (all such graphs  can be found for example in \cite{har} ,  pages 215--217). 
												
(iii)  	Let $C_k: x_1,x_2,..,,x_k,x_1$ be the unique cycle in $G$ and $f$ a $\gamma_R$-function on $G$. 
        Since $V_2^f$ is a dominating set of $G$	and $|pn[v,V_2^f]| \geq 3$ for each $v \in V_2^f$ (by Observation \ref{pn3}), 
				$|V_2^f| \leq  2$ and $k \leq 6$. 	If 	$|V_2^f|=1$ then $G = K_{1,7}+e \not \in \mathcal{R}_{UVR}$. 
												So, let without loss of generality, $V_2^f = \{x_1,x_i\}$. Clearly neither $x_1$ nor $x_i$ has more than $2$ leaves as neighbors. 
												If $k=6$ then $i=4$ and
												either each of   $x_1$ and $x_4$ have exactly one leaf as a neighbor or one of  $x_1$ and $x_4$ is adjacent to 
												$2$ leaves.  In both cases $G \not \in \mathcal{R}_{UVR}$. 
												If $k=5$ then $i \in \{3,4\}$ and one of     $x_1$ and $x_i$ has $2$ leaves as neighbors. But then $G \not \in \mathcal{R}_{UVR}$. 
												Let $k=4$. Now $x_1$ and $x_i$ has $2$ leaves as neighbors. If $i \in \{2,4\}$ then $G \not \in \mathcal{R}_{UVR}$. 
												Thus $i=3$ and $G$ is the graph depicted in Figure \ref{fig:UVR}. Clearly $G  \in \mathcal{R}_{UVR}$.   
\end{proof}

We conclude with three open problems. 

\begin{problem} \label{uni}
Characterize all unicyclic graphs that are in  $\mathcal{R}_{UVR}$. 
\end{problem}

Recall that all cycles in $\mathcal{R}_{UVR}$ are $C_{3k}$, $k \geq 1$. 

\begin{problem} \label{size}
For any pair of positive integers $n$ and $k \leq \frac{2}{3} n$
find the maximum integer $s(n, k)$ such that there is an $n$-order  graph $G \in \mathcal{R}_{UVR}$ 
with $\gamma_R(G) =k$ and $|E(G)| = s(n,k)$. 
\end{problem}

Liu and Chang \cite{lc1} proved that if  $G$ is a graph with $\delta(G) \geq 3$  then $\gamma_R(G) \leq 2n/3$. 
By Proposition \ref{3v2} we have $\gamma_R(G) < 2n/3$ when $G \in \mathcal{R}_{UVR}$ and $\delta(G) \geq 3$.  
 So, the following problem naturally arises. 

\begin{problem} \label{three}
Find an attainable constant upper bound for $\gamma_R(G)/|V(G)|$ on all connected graphs
 $G \in \mathcal{R}_{UVR}$ with $\delta(G) \geq 3$.  
\end{problem}

\end{document}